\documentclass[letterpaper,12pt]{amsart}

\usepackage{amssymb,soul}
\usepackage{bm}
\usepackage{graphicx}
\usepackage{enumerate}
\usepackage{psfrag}
\usepackage{color}
\usepackage{hyperref}

\newcommand{\excise}[1]{}%{$\star$\textsc{#1}$\star$}

\newtheorem{thm}{Theorem}[section]

\newtheorem{cor}[thm]{Corollary}
\newtheorem{prop}[thm]{Proposition}

\theoremstyle{definition}
\newtheorem{example}[thm]{Example}

\newtheorem{defn}[thm]{Definition}

\newtheorem{obs}[thm]{Observation}

\numberwithin{equation}{section}

\newcommand{\ring}[1]{\ensuremath{\mathbb{#1}}}

\newcommand\RR{\ring{R}}

\newcommand\ZZ{\ring{Z}}

\renewcommand\implies{\Rightarrow}

\newcommand{\bfd}{\mathbf{d}}
\newcommand{\bfe}{\mathbf{e}}

\newcommand{\bfi}{\mathbf{i}}
\newcommand{\bfj}{\mathbf{j}}

\newcommand{\cala}{\mathcal{A}}

\newcommand{\calc}{\mathcal{C}}
\newcommand{\cald}{\mathcal{D}}

\DeclareMathOperator{\rank}{rank}
\DeclareMathOperator\link{link}
\DeclareMathOperator\ground{ground}
\DeclareMathOperator\facet{facet}
\DeclareMathOperator\face{face}
\newcommand{\triple}[3]{\begin{smallmatrix}#1 \\ #2 \\ #3\end{smallmatrix}}

\usepackage{tikz}
\usetikzlibrary{decorations.markings}
\tikzstyle{vertex}=[circle, draw, inner sep=0pt, minimum size=6pt,fill=black]
\tikzstyle{ghost}=[circle, draw, inner sep=0pt, minimum size=6pt]
\newcommand{\vertex}{\node[vertex]}
\newcommand{\ghost}{\node[ghost]}

\DeclareMathOperator{\cone}{\textnormal{cone}}
\begin{document}%%%%%%%%%%%%%%%%%%%%%%%%%%%%%%%%%%%%%%%%%%%%%%%%%%%%%%%%
%%%%%%%%%%%%%%%%%%%%%%%%%%%%%%%%%%%%%%%%%%%%%%%%%%%%%%%%%%%%%%%%%%%%%%%%

\mbox{}
%\vspace{-2ex}%-1.1743pt}
\title[Unimodular hierarchical models and their Graver bases]{Unimodular hierarchical models\\ and their Graver bases\qquad}

\author{Daniel Irving Bernstein}
\address{Mathematics\\North Carolina State University\\Box 8205\\Raleigh, NC 27695}
\email{dibernst@ncsu.edu}

\author{Christopher O'Neill}
\address{Mathematics\\University of California, Davis\\One Shields Ave\\Davis, CA 95616}
\email{coneill@math.ucdavis.edu}

\date{\today}

\begin{abstract}
	Given a simplicial complex whose vertices are labeled with positive integers,
	one can associate a vector configuration whose corresponding toric variety
	is the Zariski closure of a hierarchical model.
	We classify all the vertex-weighted simplicial complexes that give rise to
	unimodular vector configurations.
	We also provide a combinatorial characterization of their Graver bases.
\end{abstract}
\maketitle

% MSC Subject Classifications: 

% \setcounter{tocdepth}{1}
% \tableofcontents

%%%%%%%%%%%%%%%%%%%%%%%%%%%%%%%%%%%%%%%%%%%%%%%%%%%%%%%%%%%%%%%%%%%%%%%%%
\section{Introduction}\label{s:intro}%%%%%%%%%%%%%%%%%%%%%%%%%%%%%%%%%%%%
%raggedbottom%%%%%%%%%%%%%%%%%%%%%%%%%%%%%%%%%%%%%%%%%%%%%%%%%%%%%%%%%%%%

A \emph{hierarchical model} consists of all joint probability distributions
on discrete random variables $X_1, \dots, X_n$
where each $X_i$ has $\bfd_i$ states, and the $X_i$s satisfy certain
interactions specified by a simplicial complex $\calc$.
Geometrically, a hierarchical model can be viewed as the intersection 
of the probability simplex with a toric variety whose monomial 
parametrization matrix $\cala_{\calc,\bfd}$ is determined by 
a simplicial complex $\calc$ with ground set~$[n]$
and an integer vector $\bfd \in \ZZ_{\ge 2}^n$.  
Since any hierarchical model is a discrete exponential family
with design matrix $\cala_{\calc,\bfd}$, 
it is important to study properties of 
the lattice $\ker_{\ZZ} \cala_{\calc,\bfd}$.
For example, the goodness of fit of a hierarchical model 
specified by $(\calc,\bfd)$ can be tested using a Markov basis 
of this lattice \cite{diaconis-sturmfels1998}.

We are interested in properties of the matrix $\cala_{\calc,\bfd}$ 
that are useful for computations in algebraic statistics.
Perhaps the strongest property studied in this context is unimodularity, 
and this property is desirable for a number of reasons.
When $\cala_{\calc,\bfd}$ is unimodular,
it is easy to solve the integer programs that arise
when evaluating whether individual entries
of a data table are secure \cite{sullivant2006} or when doing 
sequential importance sampling \cite{chen2006sequential}.
Additionally, Graver bases and Markov bases can be easily computed 
for unimodular matrices, and efficient generation of 
Markov basis elements is an important problem in algebraic statistics.
For more about Markov bases and their uses, see the books 
\cite{aoki2012markov,drton2008lectures,sturmfels}.

It is also worth noting that when $\cala_{\calc,\bfd}$ is unimodular, 
the semigroup spanned by its columns is saturated 
i.e.\ $\cala_{\calc,\bfd}$ is \emph{normal}.
Normality was identified by Rauh and Sullivant as 
a key property of hierarchical models for applying
the toric fiber product to construct Markov bases. 
There has been much recent work aiming to understand normality of $\cala_{\calc,\bfd}$.
Bruns, Hemmecke, Hibi, Ichim, K\"oppe, Ohsugi, and S\"oger classified which values of $\bfd$ give rise to
a normal $\cala_{\calc,\bfd}$ when $\calc$ is the boundary of a simplex \cite{Bruns2011,ohsugi-hibi2007}.
For \emph{binary} heirarchical models (that is, when $\bfd = \bf 2$),
Sullivant showed that when $\calc$ is a graph,
$\cala_{\calc,\bfd}$ is normal if and only if 
$\calc$ is free of $K_4$ minors \cite{sullivant2010}.
Sullivant and the first author also classified all binary complexes $\calc$
on up to six vertices for which $\cala_{\calc,\bfd}$ 
is normal~\cite{normalbinaryhierarchical}.
Although several conditions ensuring the normality of $\cala_{\calc,\bfd}$
have been identified, a complete characterization still seems out of reach.  

The main contributions of this paper are twofold.  
First, we classify all unimodular hierarchical models (Theorem~\ref{t:completeunimodular}), extending the existing classification for unimodular binary hierarchical models \cite{unimodularbinaryhierarchical}.  
Second, we give a combinatorial description of the Graver
basis of any unimodular hierarchical model (Corollary~\ref{c:allgraverbases}), 
one that can be leveraged to obtain an algorithm for efficiently 
generating random Graver basis elements.
In developing these results, we identify a new matrix operation which preserves unimodularity (Proposition~\ref{p:lawrenceghostunimodular}).
Similar operations, such as Lawrence lifts, have played a crutial role in identifying pairs $(\calc,\bfd)$ that yield a normal or unimodular matrix $\cala_{\calc,\bfd}$.  

%%%%%%%%%%%%%%%%%%%%%%%%%%%%%%%%%%%%%%%%%%%%%%%%%%%%%%%%%%%%%%%%%%%%%%%%%
\section{Background}\label{s:background}%%%%%%%%%%%%%%%%%%%%%%%%%%%%%%%%%
%raggedbottom%%%%%%%%%%%%%%%%%%%%%%%%%%%%%%%%%%%%%%%%%%%%%%%%%%%%%%%%%%%%

A simplicial complex $\calc$ is a pair $(V,\mathcal{F})$
where $V$ is a finite set and $\mathcal{F}$ is a set of subsets of $V$
such that if $G \subset F$ and $F \in \mathcal{F}$,
then $G \in \mathcal{F}$.
The set $V$ is called the \emph{ground set} of $\calc$ and
each $F \in \mathcal{F}$ is called a \emph{face} of $\calc$.
Inclusion-wise maximal faces of $\calc$ are called \emph{facets}.
We will use the notation
$\ground(\calc)$, $\face(\calc)$ and $\facet(\calc)$ to 
denote the ground set, faces and facets of a simplicial complex $\calc$, respectively.

An \emph{HM pair} is a pair $(\calc,\bfd)$ where
$\calc$ is a simplicial complex on some ground set $V$ 
and $\bfd \in \mathbb{Z}_{\ge 2}^V$ is an integer weighting of $V$.  
Associated to each HM pair with ground set $[n]$
is a statistical model (i.e. a subset of a probability simplex)
called a \emph{hierarchical model}.
It consists of the joint probability distributions
on discrete random variables $X_1,\dots,X_n$ where each $X_i$ has $d_i$ states
and the joint probabilities satisfy various relationships 
specified by $\calc$.
These relationships are given by polynomials which generate a toric ideal.
The matrix defining the monomial parameterization of this toric ideal is denoted $\cala_{\calc,\bfd}$, whose construction we now describe.

Fix an HM pair $(\calc,\bfd)$ with ground set $[n]$.
Define $\bfd_F = [d_{i_1}-1]\times \dots \times [d_{i_k}-1]$ for each nonempty face $F = \{i_1,\dots,i_{k}\}$ of~$C$, and define $\bfd_\emptyset = \{1\}$.
Write $\RR^{\bfd_{F}}$ for the vector space with coordinates 
indexed by $\bfj \in \bfd_F$ (whose coordinates are in turn 
indexed by the vertices of $F$).
For $\bfi \in [d_1] \times \dots \times[d_n]$, define 
$a^\bfi \in \bigoplus_{F \in \face(C)} \RR^{\bfd_{F}}$
such that 
\[
    a^{\bfi}_{F,\,\bfj} = \left\{\begin{array}{ll}
    1 & \text{whenever } F = \emptyset \text{ or } \bfi_k = \bfj_k \text{ for each } k \in F \\
    0 & \text{otherwise}
    \end{array}\right.
\]
and let $\cala_{\calc,\bfd}$ denote the matrix with columns $a^{\bfi}$ as $\bfi$ ranges over $[d_1] \times \dots \times[d_n]$.

\begin{example}\label{e:hmpair}
    The simplicial complex $\calc$ with ground set $[3]$ and facets $\{1,2\}$ and $\{2,3\}$ has faces $\emptyset, \{1\}, \{2\}, \{3\}, \{1,2\}, \{2,3\}$.
    The matrix $\cala_{\calc,\bfd}$ for $\bfd = (3,2,2)$ is displayed in Figure~\ref{f:hmpair} with row and column labels.  Note that $\bfd_F$ is a singleton unless $F$ contains $1$.
\end{example}

\begin{figure}
\begin{center}
$\begin{array}{cc|cccccccccccc|}
    \empty &\empty &\triple{1}{1}{1}& \triple{1}{1}{2}& \triple{1}{2}{1}& \triple{1}{2}{2}
    & \triple{2}{1}{1}& \triple{2}{1}{2}& \triple{2}{2}{1}& \triple{2}{2}{2}
    & \triple{3}{1}{1}& \triple{3}{1}{2}& \triple{3}{2}{1}& \triple{3}{2}{2} 
    \\[0.1in]
    \hline
    \emptyset & \cdot &1&1&1&1&1&1&1&1&1&1&1&1 \\
    \hline
    \{1\} & 1 &1&1&1&1&0&0&0&0&0&0&0&0 \\
    & 2 &0&0&0&0&1&1&1&1&0&0&0&0 \\
    \hline
    \{2\}&1 &1&1&0&0&1&1&0&0&1&1&0&0\\
    \hline
    \{3\}&1 &1&0&1&0&1&0&1&0&1&0&1&0\\
    \hline
    \{1,2\} & (1,1) &1&1&0&0&0&0&0&0&0&0&0&0 \\
    & (2,1) &0&0&0&0&1&1&0&0&0&0&0&0 \\
    \hline
    \{2,3\} & (1,1) &1&0&0&0&1&0&0&0&1&0&0&0 \\
    \hline
\end{array}$
\end{center}
\caption{The matrix $\cala_{\calc,\bfd}$ for the HM pair $(\calc,\bfd)$ in Example~\ref{e:hmpair}.}
\label{f:hmpair}
\end{figure}

We pause to note that other definitions of the matrix 
$\cala_{\calc,\bfd}$ appear in the literature.
However, they all share a common kernel over the integers, and thus define 
the same toric ideal and the same hierarchical model.
The definition of $\cala_{\calc,\bfd}$ presented here has the added benefit of being full row-rank \cite[Proposition~1]{sullivant2010}.

Given an integer matrix $A \in \ZZ^{d\times r}$,
each integer vector in the integer kernel of $A$ can be written as
$u = u^+ - u^-$ where $u^+,u^-$ have disjoint support and nonnegative entries.
The \emph{Graver basis of $A$} is the set of all $u \in \ker_\ZZ A$
such that there is no $v \in \ker_\ZZ A$ with $v \neq u$ and $v^+ \le u^+$ and $v^-\le u^-$.
In other words, $u$ cannot be expressed as a \emph{conformal sum} of $v,w \in \ker_\ZZ A$.
If $u \in \ker_\ZZ A$ has relatively prime entries and has minimal support among
elements of $\ker_\ZZ A$, then we say that $u$ is a \emph{circuit of $A$}.
Equivalently, the entries of $u$ are relatively prime and the support of $u$
is a circuit in the matroid underlying the columns of $A$.
It is easy to see that the set of circuits is a subset of the Graver basis.
For unimodular matrices (definition below),
the converse is true (Proposition 8.11 in \cite{sturmfels}) but this fails in general.

\begin{defn}\label{def:unimodular}
    A matrix $A \in \ZZ^{d\times n}$ is \emph{unimodular} if the equivalent conditions below~hold.  
    \begin{enumerate}[(a)]
    \item\label{item:minors}
    Let $A'$ be any matrix obtained from $r := \rank A$ linearly independent rows of $A$.
    Then there exists some $\lambda$ such that each $r \times r$ minor of $A'$ is $0$ or $\pm \lambda$.
    
    \item\label{item:graver}
    The Graver basis of $A$ contains only $\{0,1,-1\}$-vectors.
    
    \item\label{item:polyhedron}
    For any $b$ in the affine semigroup generated by the columns of $A$,
    the polyhedron $P_{A,b} = \{x \in \RR^s : Ax = b, x \ge 0\}$ has all integral vertices.
    \end{enumerate}
\end{defn}

See \cite{unimodularbinaryhierarchical} for an explanation as to why 
the conditions in Definition \ref{def:unimodular} are equivalent.
Definition~\ref{def:unimodular} can be seen as a kernel-invariant generalization
of the familiar definition of unimodularity for square integer matrices.
Indeed, if $A$ is unimodular, then it is easy to construct a full-rank rational matrix $B$
such that $\ker_\ZZ A = \ker_\ZZ B$ and every full-rank square submatrix of $B$
has determinant $\pm 1$.
To this end, Part~\eqref{item:polyhedron} of Definition~\ref{def:unimodular}
implies that when $A$ is unimodular,
integer programming problems over polyhedra $P_{A,b}$
can be solved via linear relaxation.
Unimodularity (as in Definition~\ref{def:unimodular}) can also be viewed 
as a generalization of total unimodularity 
(which requires that every minor is either 0 or $\pm 1$).  
That said, unimodularity is an invariant of the kernel, 
rather than just an invariant of the matrix, 
and the same does not hold for total unimodularity.  

Fix a simplicial complex $\calc$ with ground set $V$ and some $v \in V$.  
We denote by $\calc\setminus v$ the complex obtained by deleting $v$ 
from $V$ and each face of $\calc$ containing $v$.
The \emph{link of $\calc$ about $v$}, denoted $\link_v(\calc)$,
has ground set $V \setminus v$ and a face $F$ 
whenever $F \cup v$ is a face of $\calc$.
The \emph{Alexander dual of $\calc$}, denoted $\calc^*$,
is the complex on the same ground set as $\calc$
whose facets are the complements of the minimal non-faces of $\calc$.

An HM pair $(\cald,\bfd')$ is a \emph{minor} of $(\calc,\bfd)$ 
if $\cald$ can be obtained from $\calc$ via a (possibly empty) sequence of 
vertex deletions and vertex links and $d_v' \le d_v$ 
for every vertex of $\cald$.
The pair $(\calc,\bfd)$ is said to be \emph{unimodular}
if the matrix $\cala_{\calc,\bfd}$ is unimodular, 
and \emph{binary} if $\bfd = {\bf 2}$.  
In view of Proposition \ref{prop:minorClosed} below,
we say that an HM pair is \emph{minimally nonunimodular}
if it is not unimodular but every minor is.

\begin{prop}\label{prop:minorClosed}
    Minors of unimodular HM pairs are unimodular.
\end{prop}
\begin{proof}
    See Propositions~7.1~and~7.5 in \cite{unimodularbinaryhierarchical}.
\end{proof}

% One of our main results is the complete list of all minimally nonunimodular HM pairs.

The $n$-dimensional simplex will be denoted by $\Delta_n$
and its boundary complex by $\partial \Delta_n$.
The disjoint union of an $m$- and $n$-simplex will be 
denoted $\Delta_m \sqcup \Delta_n$, and its Alexander dual 
will be denoted $D_{m,n}$.
Given a simplicial complex $\calc$, we say $v \in \ground(\calc)$ is 
\begin{itemize}
\item 
a \emph{cone vertex} if it appears in every facet of $\calc$,

\item 
a \emph{ghost vertex} if it does not appear in any facet of $\calc$, and

\item 
and a \emph{Lawrence vertex} if its complement in the ground set of $\calc$ is a facet.

\end{itemize}
We denote by $\cone(\calc)$ and $G \calc$ the complex obtained 
by adding a cone vertex and ghost vertex to $\calc$, respectively, 
and we denote by $\Lambda \calc$ the complex obtained by adding 
a Lawrence vertex $v$ to $\calc$ such that $\link_v(\calc) = \calc$.
Iterated application of each aforementioned operation 
will be denoted by superscript.  
For example, $G^5 \calc$ denotes the complex obtained 
by adding five ghost vertices to $\calc$.  

Sullivant and the first author gave a complete characterization 
of the set of simplicial complexes $\calc$ 
such that the binary HM pair $(\calc,{\bf 2})$ is unimodular \cite{unimodularbinaryhierarchical}.
This characterization can be expressed constructively (Theorem~\ref{t:binaryConstructive})
or in terms of forbidden minors (Theorem~\ref{t:binaryMinors}).  
One of our main results extends Theorems~\ref{t:binaryConstructive} 
and~\ref{t:binaryMinors} to non-binary HM pairs, and provides 
a combinatorial description of the Graver basis
of the matrix $\cala_{\calc,\bfd}$ associated to any unimodular HM pair.

\begin{defn}
    A simplicial complex $\calc$ is said to be \emph{nuclear} if
    $\calc$ can be built up from $\Delta_m$, $\Delta_m \sqcup \Delta_n$ or $D_{m,n}$
    by applying a sequence of $\cone$, $G$ and $\Lambda$ operations.
    The complex that $\calc$ is built up from is called the \emph{nucleus} of $\calc$.
\end{defn}

\begin{thm}[{\cite[Theorem~6.3]{unimodularbinaryhierarchical}}]\label{t:binaryConstructive}
    The matrix $\cala_{\calc,\mathbf{2}}$ is unimodular if and only if $\calc$ is a nuclear simplicial complex. 
\end{thm}

\begin{thm}[{\cite[Theorem~6.3]{unimodularbinaryhierarchical}}]\label{t:binaryMinors}
    A simplicial complex $\calc$ is nuclear if and only if
    it has no minors isomorphic to any of the following:
    \begin{enumerate}[(1)]
        \item the disjoint union of the boundary of a simplex and an isolated vertex,
        \item the boundary complex of the octahedron or its Alexander dual,
        \item the path on four vertices, or
        \item a complex on ground set $\{1,2,3,4,5\}$ with $\facet(\calc)$ equal to
        \begin{itemize}
            \item $\{\{1,2\}, \{1,5\}, \{2,3,4\}, \{3,4,5\}\}$,
            \item $\{\{1,3,4\}, \{2,3,5\}, \{2,4,5\}\}$, or
            \item $\{\{1,2\}, \{2,3,5\}, \{3,4\}, \{1,4,5\}\}$.
        \end{itemize}
    \end{enumerate}
\end{thm}

%%%%%%%%%%%%%%%%%%%%%%%%%%%%%%%%%%%%%%%%%%%%%%%%%%%%%%%%%%%%%%%%%%%%%%%%%
\section{A new unimodularity-preserving operation}\label{s:operation}
%%%%%%%%%%%%%%%%%%%%%%%%%%%%%%%%%%%%%%%%%%%%%%%%%%%%%%%%%%%%%%%%%%%%%%%%%
In this section we describe a new unimodularity-preserving operation that
is key in the proof of Theorem~\ref{t:completeunimodular}.
We also describe how this operation acts on the Graver basis in the case that is relevant to us.
The results in this section are about general integer matrices and do not depend on any
specific properties of hierarchical models.

Given a matrix $A \in  \mathbb{Z}^{d \times n}$, denote by $G_qA$ and $\Lambda_pA$ the matrices
\[
    G_qA = 
    \begin{pmatrix}
        A & \dots & A
    \end{pmatrix}
    \qquad \text{ and } \qquad
	\Lambda_pA = 
    \begin{pmatrix}
	    A & 0 & \dots & 0 & 0\\
	    0 & A & \dots & 0 & 0\\
	    \vdots & \vdots & \ddots & \vdots & \vdots\\
	    0 & 0 & \dots & A & 0 \\
	    I & I & \dots & I & I
	\end{pmatrix}.
\]
Note that $G_q \cala_{\calc,\bfd} = \cala_{G \calc,(\bfd \ \ q)}$
and that $\Lambda_p\cala_{\calc,\bfd}$ has the
same kernel as $\cala_{\Lambda \calc, (\bfd \ \ p)}$.
The operations $\Lambda_2$ and $G_q$ for $q \ge 1$ are unimodularity-preserving,
in the sense that applying them to a unimodular matrix produces a unimodular matrix
(for $\Lambda_2$, this follows from \cite[Theorem~7.1]{sturmfels}).
In this section, we add a new unimodularity preserving operation to the list: inserting a ghost vertex operation immediately before a Lawrence lift (Proposition~\ref{p:lawrenceghostunimodular}).  This operation provides the last crucial step in generalizing Theorem~\ref{t:binaryConstructive} to arbitrary unimodular HM pairs.  

\begin{prop}\label{p:lawrenceghostunimodular}
    Let $A \in \mathbb{Z}^{d\times n}$ be a matrix and let $p\ge 2$ be a fixed integer.
    Then $\Lambda_p A$ is unimodular if and only if $\Lambda_p G_q A$ is unimodular for all integers $q \ge 2$.
\end{prop}

\begin{proof}
    Assume $A$ has full row-rank $d$.
    It follows that $\Lambda_p G_q A$ and $\Lambda_p A$ do as well.
    So, it suffices to show that any maximal square submatrix of $\Lambda_p G_q A$
    has determinant $\pm 1$ or $0$ whenever any maximal square submatrix of $\Lambda_p A$ does.
    We proceed by showing that the possible
    values of a determinant of such a sub-matrix is independent of $q$.
    To this effect, we claim that the absolute value of any such determinant
    is equal to the absolute value of a determinant of the form
    \[	\begin{array}{rl}
    		\begin{array}{c}\ \\ \ \\ \ \\ \ \\ \ \vspace{-9pt}\\ \{1\} \\ \vdots \\ 
    		R \\ S \\ T \\ \vdots \\ \left[p-1\right] \end{array}
	    	& \hspace{-10pt}\left[\begin{array}{c|cccc|cccc|c|cccc}
	    		\dots & A^R_1 & 0 & \dots & 0 & A^S_1 & 0 & \dots & 0 & \dots &B_1&0&\dots& 0 \\
	    		\dots & 0 & A^R_2 & \dots & 0 & 0 & A^S_2 & \dots & 0 & \dots &0&B_2&\dots& 0 \\
	    		\dots & \vdots & \vdots & \ddots & \vdots & \vdots & \vdots & \ddots & \vdots & \dots &\vdots&\vdots&\ddots&\vdots \\
	    		\dots & 0 & 0 & \dots & A^R_{p-1}& 0 & 0 & \dots & A^S_{p-1} & \dots &0&0&\dots&B_{p-1}\\ \hline
	    		\dots & 0 & 0 & \dots & 0 & 0 & 0 & \dots & 0 & \dots &0&0&\dots&0 \\
	    		\dots & 0 & 0 & \dots & 0 & \vdots & \vdots & \ddots & \vdots & \dots &0&0&\dots&0 \\
	    		\dots & I^R_1 & I^R_2 & \dots & I^R_{p-1}& 0 & 0 & \dots & 0 & \dots &0&0&\dots&0 \\
	    		\dots & 0 & 0 & \dots & 0& I^S_1 & I^S_2 & \dots & I^S_{p-1} & \dots &0&0&\dots&0 \\
	    		\dots & 0 & 0 & \dots & 0& 0 & 0 & \dots & 0 & \dots &0&0&\dots&0 \\
	    		\dots & \vdots & \vdots & \ddots & \vdots& \vdots & \vdots & \ddots & \vdots & \dots &\vdots&\vdots&\dots&\vdots \\
	    		\dots & 0 & 0 & \dots & 0& 0 & 0 & \dots & 0& \dots &0&0&\dots&0
	    	\end{array}\right]
    	\end{array}
    \]
    where in the block corresponding to each $S \subseteq [p-1]$,
    if $i \notin S$ then $A_i^S = I_i^S$ is the unique $0\times 0$ matrix,
    and if $i \in S$, then $A_i^S$ is some constant matrix $A^S$ and $I_i^S$
    is the identity matrix with the same number of columns as $A^S$.
    To see this, note that a square column submatrix of $\Lambda_pG_qA$
    will have a (possibly empty) submatrix of an identity matrix for its final columns.
    We can find the determinant by applying Laplace expansion about these columns.
    Each column of the resulting matrix is obtained by padding a column of $A$
    either with all zeros, or all zeros and a single 1.
    We move all columns of the former description to the rightmost side of the matrix
    and rearrange them to match the right hand part of the block structure shown above.
    For the rest of the block structure,
    note that these remaining columns are naturally partitioned into
    $p-1$ blocks of the form
    $\begin{pmatrix}0 & A_i^T & 0& I\end{pmatrix}^T$
    where $i = 1,\dots p-1$ and $A_i$ is a column submatrix of $A$.
    Each column $a_j$ of $A$ appears in some subset of these blocks
    which we can index by a subset of $[p-1]$.
    We can then organize the columns of this submatrix according to this subset indexing
    and then organize the bottom block of rows to get the desired block structure.
    
    Now if $S$ is a singleton,
    then every row in the block of rows labeled $S$ has
    exactly one $1$ and all the remaining entries $0$.
    Therefore, we can remove all those blocks using Laplace expansion along these rows.
    Using the identity matrices in the bottom half of our matrix,
    we can apply row operations to turn each row-block of the form
    $\begin{pmatrix}0 & \cdots & 0 & A^S\end{pmatrix}$
    into $\begin{pmatrix}-A^S & \cdots & -A^S & 0\end{pmatrix}$.
    This leaves several columns that have a single $1$ and all other entries $0$.
    Applying Laplace expansion along these columns leaves the matrix
    \[
    	\left[\begin{array}{c|ccc|ccc|c|cccc}
	    	\cdots & A^R & \cdots & 0 & A^S & \cdots & 0 & \cdots &B_1&\cdots& 0&0 \\
	    	\vdots & \vdots & \ddots & \vdots & \vdots & \ddots & \vdots & \vdots &\vdots&\ddots&\vdots &\vdots \\
	    	\cdots & 0 & \cdots & A^R & 0 & \cdots & A^S & \cdots &0&\cdots& B_{p-2}& 0 \\
	    	\cdots & -A^R & \cdots & -A^R& -A^S & \cdots & -A^S & \cdots &0&\cdots & 0 &B_{p-1}\\ 
	    \end{array}\right].
    \]
    Note that each $A^S$ is either the $0 \times 0$ matrix, or a column submatrix of $A$ (possibly with repeated columns).
    So the set of possible nonzero determinants is independent of $q$.
\end{proof}

We conclude this section by demonstrating how to recover the Graver basis of $\Lambda_3 G_q A$ from the Graver basis of $\Lambda_3 A$, for use in Corollary~\ref{c:allgraverbases}.  
% \[
%     \Lambda_3 G_q A = \left[\begin{array}{cccccccccccc}
%     I &   &        &   & I &   &        &   & I &   &        &   \\
%       & I &        &   &   & I &        &   &   & I &        &   \\
%       &   & \ddots &   &   &   & \ddots &   &   &   & \ddots &   \\
%       &   &        & I &   &   &        & I &   &   &        & I \\
%     % A & A & \cdots & A &   &   &        &   &   &   &        &   \\
%       &   &        &   & A & A & \cdots & A &   &   &        &   \\
%       &   &        &   &   &   &        &   & A & A & \cdots & A \\
%     \end{array}\right]
% \]
In what follows, we write each $v \in \ker \Lambda_3 G_q A$ in the form $v = (a, b, c)$ with $a, b, c \in \ker G_q A$.  Additionally, we let $v_i = (a_i, b_i, c_i) \in \ZZ^{3q}$ denote the restriction of $v$ to coordinates corresponding to the $i$-th column of $A$ (in particular, these columns are not sequential above).  

\begin{prop}\label{p:lawrenceghostgraver}
    Let $A \in \mathbb{Z}^{d\times n}$ be a matrix and fix $q \ge 1$.
    Assume $\Lambda_3 G_q A$ is unimodular.
    Then the Graver basis of $\Lambda_3 G_q A$ consists of vectors $v = (a, b, c)$ obtained in the following ways (up to scaling by $-1$ and permutation of $a$, $b$ and $c$): 
    \begin{enumerate}[(a)]
    \item 
    for some $i \le n$ and $j, k \le q$, $a_{i,j} = b_{i,k} = 1$ and $a_{i,k} = b_{i,j} = -1$ are the only nonzero entries in $v$;
    
    % \item 
    % for some $i \le n$ and $j, k, \ell \le q$, $a_{i,j} = b_{i,\ell} = c_{i,k} = 1$ and $a_{i,k} = b_{i,j} = c_{i,\ell} = -1$ are the only nonzero entries in $v$;
    
    \item 
    for every $i \le n$, each vector $a_i$, $b_i$, and $c_i$ has at most 1 nonzero entry, and writing $a_i'$, $b_i'$ and $c_i'$ for the sum of the entries of $a_i$, $b_i$, and $c_i$ respectively, the vector $(a_1', \ldots, a_n', b_1', \ldots, b_n', c_1', \ldots, c_n')$ lies in the Graver basis of $\Lambda_3 A$; or
    
    \item 
    for some $i \le n$ and $j, k \le q$, $a_{i,j} = 1$, $a_{i,k} = -1$, $b_{i,j} = -1$, and $c_{i,k} = 1$, and the vector $v' = (a', b', c')$ with all coordinates the same as in $v$, aside from
    \[
        a_{i,j}' = a_{i,k}' = 0, \quad c_{i,j}' = 1, \quad \text{and} \quad c_{i,k}' = 0
    \]
    % $(a - a_i, b - b_i + c_i, c - c_i + b_i)$
    also lies in the Graver basis of $\Lambda_3 G_q A$.  
    
    \end{enumerate}
\end{prop}

We clarify the statement of Proposition \ref{p:lawrenceghostgraver} with an example
before giving the proof.

\begin{example}\label{e:lawrenceghostgraver}
    Consider the matrix 
    \[
        A = \left[\begin{array}{ccc}
        1 & 1 & 0 \\
        0 & 1 & 1
        \end{array}\right],
    \]
    whose Graver basis consists of $v = \bfe_1 - \bfe_2 + \bfe_3$ and its negative.
    One can compute the Graver basis of $\Lambda_3 A$ using the software 4ti2 \cite{4ti2} to check
    unimodularity.
    It follows from Proposition \ref{p:lawrenceghostunimodular} that $\Lambda_3 G_q A$
    is unimodular for all $q$.
    Up to reordering, every vector in the Graver basis of $\Lambda_3 A$ has the form $(v,-v,0)$.  Using Proposition~\ref{p:lawrenceghostgraver}, we can obtain Graver basis vectors for $\Lambda_3 G_3 A$ in the following ways.  
    \begin{itemize}
    \item 
    Vectors of the form $(a,-a,0)$ for some $a$ in the Graver basis of $G_3 A$, yielding type~(i) vectors such as $((\bfe_1,-\bfe_1,0),(-\bfe_1,\bfe_1,0),(0,0,0))$, and type~(ii) vectors such~as
    \[
        ((\bfe_1 + \bfe_3,-\bfe_2,0),(-\bfe_1 - \bfe_3,\bfe_2,0),(0,0,0))
    \]
    obtained by ``spreading'' the vector $v$ across $G_3 A$.  Writing the latter vector in the form $(a,b,c)$ and using the notation introduced above Proposition \ref{p:lawrenceghostgraver},
    we have $a_1 = (1,0,0)$, $b_1 = (-1,0,0)$ and $c_1 = (0,0,0)$, whose entries correspond to the columns of $\Lambda_3 G_3 A$ containing the first column of $A$.   
    \item 
    Vectors obtained from another Graver basis vector $(a,b,c)$ by (up to reordering of $a$, $b$ and $c$) moving a single nonzero entry of $a$ to an unoccupied entry of $a$ corresponding to the same column of $A$, and then adding a 1 and a -1 to $c$ appropriately.  This yields Graver basis vectors such as 
    \[
        ((\bfe_1,-\bfe_1,0),(-\bfe_1,0,\bfe_1),(0,\bfe_1,-\bfe_1)) \quad \text{ and }
    \]
    \[
        ((\bfe_1 + \bfe_3,-\bfe_2,0),(-\bfe_3,\bfe_2,-\bfe_1),(-\bfe_1,0,\bfe_1))
    \]
    obtained from the each of the vectors above.  Notice that vectors obtained in this way are
    always of type~(iii), and that this process applied can be applied to each type~(i) and type~(ii) vector at most once for each column of~$A$.  
    \end{itemize}
\end{example}

\begin{proof}[Proof of Proposition \ref{p:lawrenceghostgraver}]
    Suppose $v$ lies in the Graver basis of $\Lambda_3 G_q A$.
    Since $\Lambda_3 G_q A$ is unimodular,
    the first $qn$ rows of $\Lambda_3 G_q A$ ensure that
    $\{a_{i,j}, b_{i,j}, c_{i,j}\}$ is $\{0,0,0\}$ or $\{1, 0, -1\}$ for all $i \le n$ and $j \le q$.  We first prove that $a_i$, $b_i$, and $c_i$ each have no repeated nonzero entries for $i \le n$.  To this end, suppose after appropriately scaling $v$ and relabeling $a$, $b$ and $c$ that $a_{i,j} = a_{i,k} = 1$, $b_{i,j} = -1$, and $c_{i,j} = 0$ for some $j, k \le q$.  
    Then $\{b_{i,k}, c_{i,k}\} = \{0,-1\}$, and in either case, the vector $v' = (a', b', c')$ obtained from $v$ by setting 
    \[
        a_{i,j}' = a_{i,j} + a_{i,k} = 2, \quad b_{i,j}' = b_{i,j} + b_{i,k}, \quad c_{i,j}' = c_{i,j} + c_{i,k}, \quad \text{and} \quad a_{i,k}' = b_{i,k}' = c_{i,k}' = 0,
    \]
    is a conformal sum of nonzero vectors in $\ker_\ZZ A$ if and only if $v$ is,
    which contradicts the unimodularity of $\Lambda_3 G_q A$.
    
    Next, if for all $i$, the vectors $a_i$, $b_i$, and $c_i$ each have at most one nonzero entry, then we are in case~(ii) above.  Otherwise, after appropriate scaling of $v$ and relabeling of $a$, $b$ and $c$, we have $a_{i,j} = 1$, $a_{i,k} = -1$, $b_{i,j} = -1$, and $c_{i,j} = 0$ for some $i \le n$ and $j, k \le q$.  If~$b_{i,k} = 1$, then these 4 nonzero entries form a primitive vector in the kernel of $\Lambda_3 G_q A$, so $v$ has no other nonzero entries and we are in case~(i) above.  In all remaining cases, $b_{i,k} = 0$ and $c_{i,k} = 1$, meaning we are in case~(iii) above.  
    Note that the vector $v'$ constructed from $v$ still yields 0 in the $j$-th and $k$-th rows of the identity blocks portion of $\Lambda_3 G_q A$, as well as each row of $\Lambda_3 G_q A$ consisting of copies of $A$.  As such, $v'$ still lies in the kernel of $\Lambda_3 G_q A$.   
    This completes the proof.
    % Checking that $v'$ indeed lies in the kernel of $\Lambda_3 G_q A$ completes the proof.  
\end{proof}

%%%%%%%%%%%%%%%%%%%%%%%%%%%%%%%%%%%%%%%%%%%%%%%%%%%%%%%%%%%%%%%%%%%%%%%%%
\section{The classification}\label{s:classification}
%%%%%%%%%%%%%%%%%%%%%%%%%%%%%%%%%%%%%%%%%%%%%%%%%%%%%%%%%%%%%%%%%%%%%%%%%

In this section, we present the complete classification of the unimodular HM pairs.  Just as~in the characterization of unimodular binary HM pairs that appeared in \cite{unimodularbinaryhierarchical}, our~classification comes in two forms: a recipe for constructing any unimodular HM pair, and~a list of forbidden minors.

\begin{obs}\label{o:changeComplex}
    Different HM pairs may yield the same matrix.
    Let $\calc$ be a complex on ground set $V$.
    Assume $\calc$ has a face $E$
    such that any for any facet $F$,
    $E \cap F \neq \emptyset$ implies $E \subseteq F$.
    Let $\calc'$ denote the complex on ground set $V' := (V \cup \{v_0\}) \setminus E$
    with facets
    \[
        \facet(\calc') = \{F \in \facet(\calc)  : F \cap E = \emptyset\}
        \cup \{F \cup \{v_0\} : E \subset F\}.
    \]
    Let $\bfd \in \ZZ_{\ge 2}^V$
    and define $\bfd' \in \ZZ_{\ge 2}^{V'}$
    such that $\bfd_v' = \bfd_v$ for all $v \in V \cap V'$
    and $\bfd_{v_0}' = \prod_{v \in F} \bfd_v$.
    Then $\ker_\ZZ\cala_{\calc,\bfd} = \ker_\ZZ\cala_{\calc',\bfd'}$.
\end{obs}

\begin{prop}\label{prop:minimallyNonUnimodular}
    The following HM pairs are minimally nonunimodular:
    \begin{enumerate}[(1)]
        \item\label{item:nuclear} any of the complexes listed in Theorem~\ref{t:binaryMinors} with $\bfd = {\bf 2}$,
        
        \item\label{item:triangle} $\Lambda(\Delta_0 \sqcup \Delta_0)$ with facets $\{12,23,13\}$
        and ${\bf d} = (3,3,3)$,

        \item\label{item:lawrence3} $\Lambda(\Delta_1 \sqcup \Delta_1)$ with facets
        $\{125,345,1234\}$ and ${\bf d} = (2,2,2,2,3)$,

        \item\label{item:lawrence3Delt} $\Lambda(\Delta_1 \sqcup \Delta_0)$ with facets
        $\{124,34,123\}$ and $\bfd = (2,2,3,3)$,
        
        \item\label{item:square} $D_{1,1}$ with facets $\{12,23,34,14\}$ and ${\bf d} = (2,2,2,3)$,
        
        \item\label{item:basket4} $\Lambda G(\Delta_0\sqcup \Delta_0)$ with facets $\{12,13,234\}$
        and ${\bf d} = (4,2,2,2)$,
        
        \item\label{item:lawrenceSquare} $\Lambda D_{1,1}$ with facets $\{1234,125,235,345,145\}$
        and ${\bf d} = (2,2,2,2,3)$, and
        
        \item\label{item:doubleLawrence3} $\Lambda \Lambda G (\Delta_0\sqcup \Delta_0)$ with facets $\{1234,1235,145,245\}$
        and $\bfd = (2,2,2,3,3)$.
    \end{enumerate}
\end{prop}

\begin{proof}
    \cite[Proposition~4.1]{unimodularbinaryhierarchical}
    gives minimal nonunimodularity of the complexes listed in Theorem~\ref{t:binaryMinors}.
    The Graver basis for~\eqref{item:doubleLawrence3} was too large to compute,
    but selecting a large enough random sample of the matrix's columns
    enabled 4ti2~\cite{4ti2} to produce
    Graver basis elements with entries of absolute value strictly greater than $1$.
    Computations are available on our website~\cite{BernsteinWeb2016}.
    In light of Observation~\ref{o:changeComplex},
    we see that the matrix for~\eqref{item:lawrence3} has the same integer kernel as the matrix for
    the HM pair $(\{12,23,13\},(4,4,3))$ and the matrix for~\eqref{item:lawrence3Delt}
    has the same integer kernel as the matrix for the HM pair $(\{12,23,13),(4,3,3)\}$
    which both have~\eqref{item:triangle} as a minor.
    Minimal nonunimodularity of the remaining HM pairs are given by
    \cite[Proposition~7.9]{unimodularbinaryhierarchical}.
\end{proof}

\begin{thm}\label{t:completeunimodular}
    The following are equivalent:
    \begin{enumerate}[(a)]
        \item\label{item:pairUni}
        $(\calc,\bfd)$ is unimodular;
        
        \item\label{item:pairNoMinor}
        $(\calc,\bfd)$ contains no minor isomorphic to any HM pair listed 
        in Proposition~\ref{prop:minimallyNonUnimodular}; and
        
        \item\label{item:pairConstructive}
        $\calc$ is nuclear. If $\calc$ has nucleus $D_{m,n}$, then $\bfd_v = 2$ for each Lawrence vertex $v$
        and each vertex $v$ from the nucleus $D_{m,n}$.
        Otherwise, we can choose the vertices
        of $\calc$ to make nucleus $\Delta_m \sqcup \Delta_n$ so that either
        	\begin{enumerate}[(1)]
        	    \item $\bfd_v = 2$ for each Lawrence vertex $v$, or
        	    \item $\min\{m,n\} = 0$, $\bfd_v = 2$ for the unique vertex $v$ of $\Delta_0$
        	    and $\bfd_v \le 3$ for each Lawrence vertex $v$ with equality attained at most once.
        	\end{enumerate}
    \end{enumerate}
\end{thm}

\begin{proof}
	The implication~\eqref{item:pairUni}~$\implies$~\eqref{item:pairNoMinor}
	follows from Propositions~\ref{prop:minimallyNonUnimodular} and \ref{prop:minorClosed}.
	Now, if $(\calc,\bfd)$ satisfies~\eqref{item:pairNoMinor}, then via the minors in
    \eqref{item:nuclear}, Theorems~\ref{t:binaryConstructive} 
    and~\ref{t:binaryMinors} imply $\calc$ is nuclear.
    So~\eqref{item:pairConstructive} follows once we show that the 
    remaining minors in Proposition~\ref{prop:minimallyNonUnimodular} 
    ensure $\bfd$ satisfies the necessary requirements.
    
    If $\calc$ has nucleus $D_{m,n}$, then minors~\eqref{item:square}
    and~\eqref{item:lawrenceSquare} ensure that~\eqref{item:pairConstructive} is satisfied, so assume $\calc$ has nucleus 
    $\Delta_m \sqcup \Delta_n$ and that $\calc$ has a Lawrence vertex $v$ 
    such that $\bfd_v \ge 3$.  If $\bfd_v \ge 4$,
    then minor~\eqref{item:basket4} ensures
    that in the iterative construction of $\calc$ as a nuclear complex,
    $v$ must have been added before any ghost vertices.
    Therefore $\calc$ is built up from $\Lambda^k(\Delta_m \sqcup \Delta_m)$.
    Permuting the order in which we add a Lawrence vertex
    does not change the resulting vertex-labeled complex
    so we may assume $k = 1$.  Minor~\eqref{item:lawrence3} 
    ensures we can assume without loss of generality that $n=0$.
    So at this point, we know $(\calc,\bfd)$
    has been built from an HM pair $(\calc',\bfd')$
    where $\calc' = \Lambda(\Delta_m\sqcup \Delta_0)$.
    Minor~\eqref{item:lawrence3Delt} ensures that
    $\bfd_u' = 2$ for the vertex $u$ from $\Delta_0$.
    However, we can realize this same complex where $u$ is the 
    Lawrence vertex and $v$ is the unique vertex in $\Delta_0$, 
    so we could have chosen a different set of vertices to play the role of 
    the nucleus that would not require $v$ to be added as a Lawrence vertex.
    
    To complete the proof of
    \eqref{item:pairNoMinor}~$\implies$~\eqref{item:pairConstructive}, 
    it remains to consider the case where $\bfd_v \in \{2,3\}$ 
    for each Lawrence vertex $v$.  Using a similar argument as before,
    we can see that if all the Lawrence vertices with $\bfd_v = 3$
    are added before any ghost vertices in the construction 
    of~$\calc$ as a nuclear complex,
    then by choosing different vertices to play the role of the nucleus
    we could have $\bfd_v = 2$ for all Lawrence vertices $v$.
    So assume all non-binary Lawrence vertices are added after a ghost vertex.
    Minor~\eqref{item:doubleLawrence3} ensures that at most one Lawrence
    vertex $v$ can have $\bfd_v = 3$, and minors~\eqref{item:triangle},
    \eqref{item:lawrence3}, and~\eqref{item:lawrence3Delt}
    imply the remaining conditions.
    
    It remains to show \eqref{item:pairConstructive} $\implies$ \eqref{item:pairUni}
    so assume $(\calc,\bfd)$ satisfies \eqref{item:pairConstructive}.
    If $\calc$ has nucleus $D_{m,n}$, then \cite[Theorem~7.10]{unimodularbinaryhierarchical} implies $(\calc,\bfd)$ is unimodular
    if and only if it satisfies \eqref{item:pairConstructive}.
    Since the operation $\cone$ commutes with $\Lambda$ and $G$ and
    unimodularity of $(\calc,\bfd)$ is independent of $\bfd_v$ for all cone vertices $v$,
    we may assume that $\calc$ can be obtained without using the $\cone$ operation.
    We may also assume that $G$ is never applied twice in a row,
    as doing so yields the same matrix as simply adding a single 
    ghost vertex with a larger vertex label.
    
    If, on the other hand, $\calc$ has nucleus $\Delta_m$,
    then since $\Lambda G \Delta_m = \cone^{m+1}(\Delta_0 \sqcup \Delta_0)$,
    we may restrict attention to the final case, namely where $\calc$ has nucleus $\Delta_m \sqcup \Delta_n$.
    In particular, 
    \[  
        \calc = \Lambda^{k_0} G \Lambda^{k_1} G 
        \Lambda^{k_2} G \cdots G \Lambda^{k_{l}}(\Delta_m \sqcup \Delta_n).
    \]

    % Now, unless $\bfd_v \ge 3$ for at most one vertex corresponding to a $\Lambda$ operation,
    % $(\calc,\bfd)$ has a minor isomorphic to \eqref{item:doubleLawrence3}.
    % Therefore, at most one such $v$ can satisfy $\bfd_v \ge 3$.
    % If, for such a $v$, we have $\bfd_v \ge 4$, then $(\calc,\bfd)$ 
    % has a minor isomorphic to \eqref{item:basket4}.
    % Since $(\cald,\bfd')$ is unimodular if and only if 
    % $(\Lambda \cald,(\bfd'\ \ 2))$ is unimodular, 
    % we may assume that the leftmost $\Lambda$ corresponds to the vertex $v$ 
    % such that $\bfd_v = 3$.
    
    % If $\bfd_u,\bfd_v \ge 3$ for two vertices $u,v$ from $\partial \Delta_m$,
    % then denoting by $w$ the Lawrence vertex with $\bfd_w = 3$
    % we have that $(\calc,\bfd)$ has a minor isomorphic to \eqref{item:triangle}
    % supported on $\{u,v,w\}$.
    % This leaves us with the complex $\calc = \Lambda G \Lambda^{k_1} G \Lambda^{k_2} G \dots \Lambda^{k_l} G \partial\Delta_m$
    % and $\bfd$ satisfying $\bfd_v = 3$ for the leftmost Lawrence vertex,
    % $\bfd_v = 2$ for all other Lawrence vertices,
    % and $\bfd_u \ge 3$ for at most one vertex from $\partial\Delta_m$.

    If $(\calc,\bfd)$ satisfies the first case of (\ref{item:pairConstructive})
    then unimodularity of $(\calc,\bfd)$ follows from unimodularity of $(\Delta_m \sqcup \Delta_n,\bfd')$
    and the fact that adding Lawrence vertices with label $2$
    and ghost vertices of any vertex label preserves unimodularity.
    So it only remains to show unimodularity of pairs
    $(\calc,\bfd)$ that satisfy the second case of (\ref{item:pairConstructive}).

    Let $(\calc',\bfd')$ be a hierarchical model.
    We claim that if $(\Lambda\calc',(\bfd', 3))$ is unimodular, then
    $(\Lambda G \calc',(\bfd', q, 3))$ and $(\Lambda \Lambda \calc',(\bfd', 2, 3))$ are unimodular as well.
    Indeed, unimodularity of $(\Lambda G \calc',(\bfd', q, 3))$ follows from 
    Proposition~\ref{p:lawrenceghostunimodular} and
    $\ker_\ZZ(\cala_{\Lambda G \calc',(\bfd', q, 3)}) = \ker_\ZZ(\Lambda_3 G_q \cala_{\calc',\bfd'})$
    and $\ker_\ZZ(\cala_{\Lambda \calc',(\bfd', 3)}) = \ker_\ZZ(\Lambda_3 \cala_{\calc',\bfd'})$.
    Additionally, $(\Lambda \Lambda \calc',(\bfd', 2, 3))$ has the same defining matrix as $(\Lambda \Lambda \calc',(\bfd', 3, 2))$,
    which is unimodular if and only if $(\Lambda \calc',(\bfd', 3))$ is.
    At this point, unimodularity of $(\calc,{\bfd})$ follows by induction 
    if we show unimodularity of the HM pair $(\Lambda(\Delta_m \sqcup \Delta_0),\bfe)$
    where $\bfe_v = 3$ for the Lawrence vertex and $\bfe_v = 2$ for the
    vertex of $\Delta_0$.
    Letting $p$ be the product of the vertex labels in $\Delta_m$,
    this HM pair has the same matrix as the HM pair $(\{12,13,23\},(3,2,p))$
    which was shown to be unimodular in \cite{ohsugi-hibi2007}.
\end{proof}

As a corollary of Theorem~\ref{t:completeunimodular}, we obtain 
a classification of the unimodular \emph{discrete undirected graphical models},
that is, hierarchical models whose simplicial complex 
is the clique complex of a graph.
For a given graph $G$, we let $\calc(G)$ denote the clique complex of $G$.
A \emph{suspension vertex} of $G$ is a vertex that 
shares an edge with every other vertex.
Let $S^kG$ denote the graph obtained by adding $k$ suspension vertices to $G$.

\begin{cor}\label{c:cliquecomplex}
    The matrix $\cala_{\calc(G),\bfd}$ is unimodular if and only if one of the following holds: 
    \begin{enumerate}[(a)]
    	\item $G$ is a complete graph;
        \item $G = S^kC_4$, where $C_4$ is the four-cycle and $\bfd_v = 2$
        for each vertex from $C_4$; or
        \item $G$ is obtained by gluing two complete graphs along a (possibly empty) common clique.
    \end{enumerate}
\end{cor}
\begin{proof}
    The constraints on $G$ follow immediately from \cite[Lemma~5.3]{unimodularbinaryhierarchical}.
    When $G = S^kC_4$,
    Theorem~\ref{t:completeunimodular} implies the constraints on $\bfd$.
\end{proof}

%%%%%%%%%%%%%%%%%%%%%%%%%%%%%%%%%%%%%%%%%%%%%%%%%%%%%%%%%%%%%%%%%%%%%%%%%
\section{The Graver basis of a unimodular hierarchical model}\label{s:graver}
%%%%%%%%%%%%%%%%%%%%%%%%%%%%%%%%%%%%%%%%%%%%%%%%%%%%%%%%%%%%%%%%%%%%%%%%%

In the final section of this paper, we present a combinatorial 
characterization of the Graver basis of any unimodular hierarchical model's 
defining matrix.  Following the constructive characterization of unimodular 
hierarchical models in Theorem~\ref{t:completeunimodular}, our 
characterization comes in two steps: (i) a description of the Graver basis 
of each nucleus, and (ii) a description of how to obtain the Graver basis 
of a matrix produced by one of the unimodularity preserving operations
allowed by Theorem~\ref{t:completeunimodular}\eqref{item:pairConstructive} given the Graver basis of its input matrix.  

Proposition~\ref{p:binaryCircuits} characterizes the Graver basis of each 
possible nucleus.   For any unimodular matrix, the Graver basis consists of 
vectors with entries in $\{0,1,-1\}$ and coincides with the set of circuits
(see Proposition 8.11 in \cite{sturmfels}).  
As such, the characterization presented in Proposition~\ref{p:binaryCircuits} 
simply describes the signed circuits
of the oriented matroid underlying each given hierarchical model.
We remind the reader how to construct oriented matroids 
from signed circuits and cuts of a directed graph;
for a more thorough introduction to oriented matroids, see~\cite{bjorner1999oriented}.

A \emph{signed circuit} of a directed graph $G$ is a bipartition
of the set of edges in a simple cycle
$v_1,\dots,v_n,v_1$ of the undirected graph underlying $G$
according to whether or not the edge $v_iv_{i+1}$ agrees with $G$'s orientation.
The edges that agree are called \emph{positive} while those that disagree are called \emph{negative}.
A \emph{signed bond} of $G$ is a bipartition of the set of edges in a bond (i.e. minimal cut)
of $G$ that splits $G$ into connected components $A$ and $B$
according to whether or not the edge $e$ is directed from $A$ to $B$.
The edges pointing from $A$ to $B$ are called \emph{positive} while those
pointing from $B$ to $A$ are called \emph{negative}.
The set of signed circuits and the set of signed bonds of $G$
each are the signed circuits of an oriented matroid.
These two oriented matroids are dual to each other.

Let $K_{2^{m+1},2^{n+1}}$ denote the complete bipartite graph on
partite sets $2^{[m+1]}$ and $2^{[n+1]}$.
Label the vertices in each partite set by the binary $(m+1)$- and 
$(n+1)$-tuples in $\{1,2\}^{m+1}$ and $\{1,2\}^{n+1}$, respectively.
Each edge is naturally labeled with the $(m+n+2)$-tuple
obtained by concatenating the labels of its vertices.
Let $G_1^{m,n}$ be the directed graph with underlying undirected graph 
$K_{2^{m+1},2^{n+1}}$ where all edges are directed from the partite set 
$2^{[m+1]}$ to the partite set $2^{[n+1]}$.
Let $G_2^{m,n}$ be the directed graph obtained from $G_1^{m,n}$ by
reversing the orientation of the edges whose 
$(m+n+2)$-tuple has an odd number of twos.

\begin{prop}\label{p:binaryCircuits}
	The set of signed circuits of the oriented matroid underlying
    the columns of $\cala_{\Delta_m}$ is empty.
    The signed circuits of the oriented matroid underlying
    the columns of $\cala_{\Delta_m \sqcup \Delta_n,{\bf 2}}$
    are the signed circuits of $G_1^{m,n}$.
    The signed circuits of the oriented matroid underlying
    the columns of $\cala_{D_{m,n},{\bf 2}}$
    are the signed bonds of $G_2^{m,n}$.
\end{prop}
\begin{proof}
	For the first statement, note that the matrix $\cala_{\Delta_m}$
	is square with determinant $1$ and therefore has trivial kernel. 
	For the second statement,
	recall that row operations do not affect the underling oriented matroid.
	Then note that after multiplying the appropriate rows of
	$\cala_{\Delta_m \sqcup \Delta_n,\mathbf{2}}$
    by $-1$, we obtain the vertex-arc incidence matrix of $G_1^{m, n}$.

	By Proposition 3.6 in \cite{unimodularbinaryhierarchical} and its proof,
    $\cala_{D_{m,n},{\bf 2}}$ can be obtained from a particular Gale dual $B$ of
    $\cala_{\Delta_m \sqcup \Delta_n,{\bf 2}}$ by negating the
    columns of $B$ corresponding to the binary $(m+n+2)$-tuples
    with an odd number of twos, then
    negating all negative rows of the resulting matrix.
    The oriented matroid underlying the columns of $B$ is dual
    to the oriented matroid underlying the columns of $\cala_{\Delta_m \sqcup \Delta_n,{\bf 2}}$.
    Therefore, the signed circuits
    of the oriented matroid underlying the columns of $B$
    are the signed bonds of $G_1^{m,n}$.
    On the oriented matroid level,
    the process of turning $B$ into $\cala_{D_{m,n},{\bf 2}}$
    by negating the appropriate rows and columns has the effect
    of reversing the orientation of the edges of $G_1^{m,n}$
    corresponding to binary $(m+n+2)$-tuples with an odd number of twos.
    This gives us $G_2^{m,n}$.
\end{proof}

\begin{example}\label{e:binaryCircuits}
	We illustrate Proposition~\ref{p:binaryCircuits} in the case $m = 1$ and $n = 0$.
	We can draw the relevant simplicial complexes as follows.
	Note that $D_{1,0}$ has a ghost vertex which we indicate pictorially with an open circle.
	\[
	\begin{tikzpicture}
		\node at (-2,-1){$\Delta_{1} \sqcup \Delta_0 =$};
	    \vertex (1) at (0,0)[label=left:$1$]{};
	    \vertex (2) at (0,-2)[label=left:$2$]{};
	    \vertex (3) at (1,-1)[label=right:$3$]{};
	    \path
	    	(1) edge (2)
	    ;
	\end{tikzpicture}
	\qquad\qquad
	\begin{tikzpicture}
		\node at (-2,-1){$D_{1,0}=$};
	    \vertex (1) at (0,0)[label=left:$1$]{};
	    \vertex (2) at (0,-2)[label=left:$2$]{};
	    \ghost (3) at (1,-1)[label=right:$3$]{};
	\end{tikzpicture}
	\]
	% \todaniel{Can these pictures be made a bit smaller/shorter?  There seems to be a lot of wasted space.}
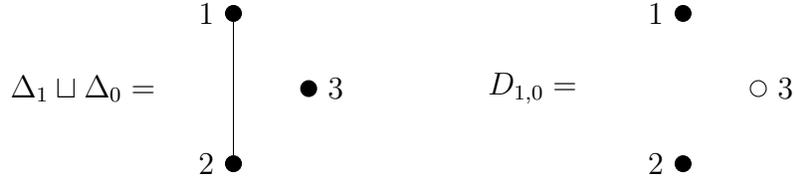
	Figure~\ref{f:binaryCircuits} depicts
	$\cala_{\Delta_{1} \sqcup \Delta_0,{\bf 2}}$ 
	and $\cala_{D_{1,0},{\bf 2}}$
	alongside the directed graphs $G_1^{1,0}$ and $G_2^{1,0}$.  
	
	The edges corresponding to the binary $3$-tuples
	$221,222,212,211$ form a cycle in $G_1^{1,0}$ with $221$ and $212$
	having positive orientation, and $222$ and $211$ both having negative orientation.
	Therefore the vector $e_{221} + e_{212} - e_{222} - e_{211}$ is in the Graver basis
	of $\cala_{\Delta_1 \sqcup \Delta_0,{\bf 2}}$.
	The edges corresponding to the binary $3$-tuples $221,211,121,111$
	form a minimal cut in $G_2^{1,0}$,
	partitioning the vertices into $\{1\}$ and its complement.
	Calling these sets $A$ and $B$ respectively,
	note that the edges corresponding to $221$ and $111$ point from $A$ into $B$
	whereas the edges corresponding to $211$ and $121$ point from $B$ into $A$.
	Therefore, the vector $e_{221} + e_{111} - e_{211} - e_{121}$ lies in 
	the Graver basis of $\cala_{D_{1,0},{\bf 2}}$.
\end{example}

We now state the obvious extension of Proposition \ref{p:binaryCircuits}
to the case where $\calc = \Delta_m \sqcup \Delta_n$
but $\bfd = {\bf 2}$ need not hold.

\begin{prop}\label{p:nonbinaryCircuits}
    Let $\calc = \Delta_m \sqcup \Delta_n$ with vertex set $\{1,\dots,m+n+2\}$
    and facets $\{1,\dots,m+1\}$ and $\{m+2,\dots,m+n+2\}$.
    Fix $\bfd \in \ZZ^{m+n+2}$ and let $G^{m,n}_\bfd$ be the complete bipartite graph
    with partite vertex sets $[d_1]\times\dots\times[d_{m+1}]$
    and $[d_{m+2}]\times\dots\times[d_{m+n+2}]$, with edges oriented so that they all point towards the same partite set.
    Label each edge by the element of $[d_1]\times[d_{m+n+2}]$
    obtained by concatenating its vertices.
    Then the signed circuits of the oriented matroid underlying the columns
    of $\cala_{\calc,\bfd}$ are the signed circuits of $G^{m,n}_\bfd$.
\end{prop}

\begin{figure}
\begin{center}
	$\begin{array}{c@{\qquad}c}
		\begin{array}{cc|cccccccc|}
		    &\empty &\triple{1}{1}{1}& \triple{1}{1}{2}& \triple{1}{2}{1}& \triple{1}{2}{2}
		    & \triple{2}{1}{1}& \triple{2}{1}{2}& \triple{2}{2}{1}& \triple{2}{2}{2}\\
		    \hline
		    &\emptyset &1&1&1&1&1&1&1&1\\
		    &\{1\} &1&1&1&1&0&0&0&0\\
		    &\{2\} &1&1&0&0&1&1&0&0\\
		    &\{3\} &1&0&1&0&1&0&1&0\\
		    &\{1,2\} &1&1&0&0&0&0&0&0\\
		    \hline
		    \\[-0.8em]
		\end{array}
		&
		\begin{array}{c@{}c|cccccccc|}
		    \\
		    \\
		    &\empty &\triple{1}{1}{1}& \triple{1}{1}{2}& \triple{1}{2}{1}& \triple{1}{2}{2}
		    &\triple{2}{1}{1}& \triple{2}{1}{2}& \triple{2}{2}{1}& \triple{2}{2}{2}\\
		    \hline
		    &\emptyset &1&1&1&1&1&1&1&1\\
		    &\{1\} &1&1&1&1&0&0&0&0\\
		    &\{2\} &1&1&0&0&1&1&0&0\\
		    \hline
		    \\[-0.8em]
		\end{array}
		\\
		\cala_{\Delta_{1} \sqcup \Delta_0,{\bf 2}}
		&
		\cala_{D_{1,0},{\bf 2}}
	\end{array}$
\end{center}
\vspace{0.05in}
\begin{center}
	\begin{tikzpicture}[x=1.3cm, y=1cm,
		every edge/.style={draw,postaction={decorate,decoration={markings,mark=at position 0.5 with {\arrow{>}}}}}]
		\node  at (-1,0){$G_1^{1,0}=$};
		\vertex (0) at (0,0)[label=left:$1$]{};
		\vertex (1) at (2,0)[label=right:$2$]{};
		\vertex (00) at (1,-2)[label=above:$11$]{};
		\vertex (01) at (1,-1)[label=above:$12$]{};
		\vertex (10) at (1,0)[label=above:$21$]{};
		\vertex (11) at (1,1)[label=above:$22$]{};
		\path
			(0) edge (00)
			(1) edge (00)
			(0) edge (01)
			(1) edge (01)
			(0) edge (10)
			(1) edge (10)
			(0) edge (11)
			(1) edge (11)
		;
	\end{tikzpicture}
	\qquad\qquad
	\begin{tikzpicture}[x=1.3cm, y=1cm,
		every edge/.style={draw,postaction={decorate,decoration={markings,mark=at position 0.5 with {\arrow{>}}}}}]
		\node  at (-1,0){$G_2^{1,0}=$};
		\vertex (0) at (0,0)[label=left:$1$]{};
		\vertex (1) at (2,0)[label=right:$2$]{};
		\vertex (00) at (1,-2)[label=above:$11$]{};
		\vertex (01) at (1,-1)[label=above:$12$]{};
		\vertex (10) at (1,0)[label=above:$21$]{};
		\vertex (11) at (1,1)[label=above:$22$]{};
		\path
			(0) edge (00)
			(1) edge (01)
			(1) edge (10)
			(0) edge (11)
			(00) edge (1)
			(01) edge (0)
			(10) edge (0)
			(11) edge (1)
		;
	\end{tikzpicture}
\end{center}
\caption{The matrices $\cala_{\Delta_{1} \sqcup \Delta_0,{\bf 2}}$ and $\cala_{D_{1,0},{\bf 2}}$, along with their respective graphs $G_1^{1,0}$ and $G_2^{1,0}$, from Example~\ref{e:binaryCircuits}.}
\label{f:binaryCircuits}
\end{figure}
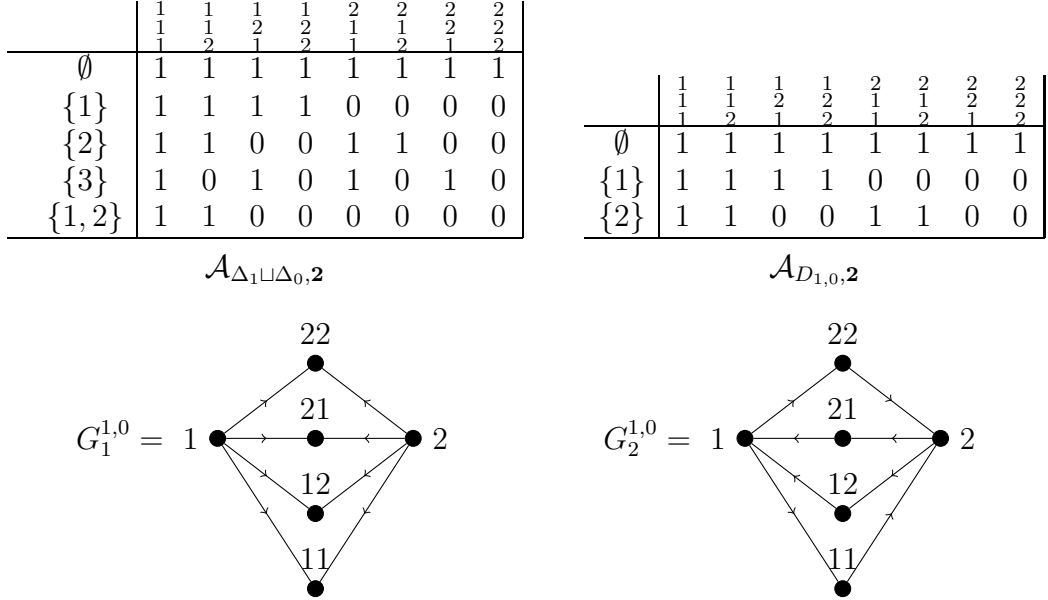

Having now characterized the Graver basis of each nucleus, 
it remains to describe how each of the operations 
$\cone(-)$, $G$, $\Lambda_2$ and $\Lambda_3$ 
affect the Graver basis of a unimodular matrix $\cala_{\calc,\bfd}$.
Adding ghost vertices changes $\cala_{\calc,\bfd}$ by
$\cala_{G \calc,(\bfd \ \ q)} = G_q(\cala_{\calc,\bfd})$.
Therefore,
every element in the Graver basis of $\cala_{G\calc,(\bfd \ \ q)}$
is either of the form
\[
	(0 \ \ \cdots \ \ 0 \ \ e_i \ \ 0 \ \ \cdots \ \ 0 \ \ -e_i \ \ 0 \ \ \cdots \ \ 0)
\]
where $e_i$ is the $i^{\rm th}$ standard basis vector,
or $\begin{pmatrix} u_1 & u_2 & \dots & u_q \end{pmatrix}$
where $u_1 + u_2 + \dots + u_q$ is a conformal sum which lies in the Graver basis
of $\cala_{\calc,\bfd}$.
Keeping in mind the representation of $\cala_{\calc,\bfd}$
given in \cite{unimodularbinaryhierarchical},
it is easy to see that the kernel of $\cala_{\cone(\calc),(\bfd \ \ q)}$
is the same as the kernel of the following matrix
\[
	\begin{pmatrix}
		\cala_{\calc,\bfd} & 0 &  \dots & 0 \\
		0 & \cala_{\calc,\bfd} & \dots & 0 \\
		\vdots & \vdots & \ddots & \vdots \\
		0 & 0 & \dots & \cala_{\calc,\bfd}
	\end{pmatrix}.
\]
The Graver basis of $\cala_{\cone(\calc),(\bfd, q)}$ consists of 
elements of the form $(0 \ \ \cdots \ \ 0 \ \ u \ \ 0 \ \ \cdots \ \ 0)$
where $u$ is in the Graver basis of $\cala_{\calc,\bfd}$.
Again, considering the representation of $\cala_{\calc,\bfd}$
given in \cite{unimodularbinaryhierarchical},
the kernel of $\cala_{\Lambda(\calc),(\bfd \ \ p)}$
is the same as the kernel of $\Lambda_p\cala_{\calc,\bfd}$.
From this it easily follows that the Graver basis of $\cala_{\Lambda(\calc),(\bfd \ \ 2)}$
consists of elements of the form $(u \ \ -u)$ where $u$ is in the graver basis
of $\cala_{\calc,\bfd}$. 

This leaves the operation $\Lambda_3$.
Theorem~\ref{t:completeunimodular} tells us that
we only need to consider $\Lambda_3$ when applied to
$\cala_{\calc,\bfd}$ where $\calc$ is nuclear with nucleus $\Delta_m \sqcup \Delta_0$
and $\bfd_v = 2$ for the vertex of the $\Delta_0$ and every Lawrence vertex.
The operation $\cone(-)$ commutes with each other operation, and $\Lambda_2$ commutes with~$\Lambda_3$, so it suffices to describe the Graver basis of complexes of the form
\[
	\calc = \Lambda_3 G \Lambda^{k_1} G_{q_1} 
        \Lambda_2^{k_2} G_{q_2} \cdots \Lambda_2^{k_l} G_{q_l} \Lambda_2^{k_{l+1}}(\Delta_m \sqcup \Delta_0).
\]
Using Proposition~\ref{p:lawrenceghostgraver}, a Graver basis 
for the matrix corresponding to $\calc$ can be obtained inductively, 
starting with the Graver basis for the matrix corresponding to the 
labeled complex $\Lambda_3 (\Delta_m \sqcup \Delta_0)$, whose 
matrix is the same as for the HM pair $(\{12,13,23\}, (3,2,p))$.  
% the proof of Theorem~\ref{t:completeunimodular}

\begin{cor}\label{c:allgraverbases}
    The Graver basis of any unimodular HM pair can be obtained 
    by way of Propositions~\ref{p:binaryCircuits} and \ref{p:nonbinaryCircuits} and the discussion thereafter.  
\end{cor}

\section*{Acknowledgments}
This work began at the 2016 AMS Mathematics Research Communities in Algebraic Statistics
which was supported by the National Science Foundation under Grant Number DMS 1321794.
Daniel Irving Bernstein was partially supported by the US National Science
Foundation (DMS 0954865) and the David and Lucille Packard Foundation.

{\footnotesize
\bibliography{normal}}
\bibliographystyle{plain}

%%%%%%%%%%%%%%%%%%%%%%%%%%%%%%%%%%%%%%%%%%%%%%%%%%%%%%%%%%%%%%%%%%%%%%%%%
\end{document}